\documentclass[11pt]{article}

\usepackage{geometry} 
\usepackage{graphicx} 

\usepackage{booktabs} 
\usepackage{array} 
\usepackage{paralist} 
\usepackage{verbatim} 
\usepackage{subfig} 

\usepackage{amsthm}
\usepackage{amsmath}
\usepackage{amssymb}
\usepackage{graphicx}
\usepackage{amscd}
\usepackage[all, cmtip]{xy}
\usepackage{mathtools}

\newtheorem{thm}{Theorem}[section]
\newtheorem{lem}[thm]{Lemma}
\newtheorem{prop}[thm]{Proposition}

\newtheorem{cor}[thm]{Corollary}
\newtheorem{remark}{Remark}[section]
\newtheorem{example}{Example}

\theoremstyle{defn}

\numberwithin{equation}{section}

\def \bt {\begin{tabular}{l}}
\def \et {\end{tabular}}

\begin{document}

\title{Relative Logarithmic Cohomology and Nambu Structures of Maximal Degree}

\author{Konstantinos Kourliouros\footnote{
Avenida Trabalhador Sancarlense, 400 - Centro, S\~ao Carlos - SP, 13566-590, \newline 
email: k.kourliouros@gmail.com}\\
ICMC-USP}

\maketitle

\begin{abstract}
We present local classification results for isolated singularities of functions with respect to a Nambu structure (multi-vector field) of maximal degree, in a neighbourhood of a smooth point of its degeneracy hypersurface. The results depend on a logarithmic version of the Brieskorn-Sebastiani theorem, which guarantees the finiteness and freeness of the corresponding deformation module. This relates the functional moduli of the classification problem with the integrals of logarithmic forms along the vanishing cycles of the complement of the Milnor fibers of the restriction of the function on the degeneracy hypersurface of the Nambu structure, inside the Milnor fibers of the function itself. 
\end{abstract}

\medskip 




\maketitle

\section{Introduction}
\label{}

Nambu structures are important objects in mathematical physics and they have been introduced by Y. Nambu \cite{N}, as a natural generalisation of the classical Poisson structures. After Nambu, their exact mathematical formulation has been given by L. Takhtajan in \cite{T}. Briefly, a Nambu structure of degree $r$ on an $n$-dimensional manifold $X$ (where $2\le r \le n$) can be given by an $r$-fold skew-symmetric operator (so called Nambu, or Nambu-Poisson bracket) on the algebra of functions $\mathcal{O}_X$ on $X$:
\[\{\cdot,\cdots,\cdot\}:\mathcal{O}_X^{\otimes r}\longrightarrow \mathcal{O}_X,\]
which satisfies the Leibniz rule in each of its entries (is a derivation), and an identity (usually called Fundamental, or Filippov identity), which generalises the Jacobi identity of Poisson structures (c.f. \cite{T} for exact statement, but it will not be needed here). In particular, for $r=2$ we obtain Poisson structures, and for maximal degree $r=n$ we obtain an operator which assigns to an $n$-tuple of functions $f_1,\cdots, f_n$, the function:
\[\{f_1,\cdots,f_n\}:= \Pi(df_1,\cdots,df_n),\]
where $\Pi \in \Gamma(\wedge^n(TX))$ is a multivector field of maximal degree $n$ (also called Nambu tensor c.f. \cite{DN1}). In local coordinates $x=(x_1,\cdots,x_n)$ for which $\Pi=\sigma(x)\partial_{x_1}\wedge \cdots \wedge \partial_{x_n}$, for some function $\sigma$, the Nambu bracket is completely determined by the equation:
\[\{x_1,\cdots,x_n\}:=\Pi(dx_1,\cdots,dx_n)=\sigma(x),\]
so that:   
\[\{f_1,\cdots,f_n\}=\sigma(x)\det(\frac{\partial f_i}{\partial x_j}).\]
We remark that in this case of maximal degree, the fundamental identity is void (trivially satisfied), and for $r=n=2$ we obtain the classical Poisson structures on 2-dimensional manifolds. Notice moreover that the singular locus of the Nambu structure is the hypersurface $\Sigma=\{\sigma=0\}$ where the Nambu tensor vanishes identically. In this case the Nambu structure is called singular, and it is called non-singular (or non-degenerate) otherwise (i.e. when the function $\sigma$ is a unit).

Local classification of Nambu structures of maximal degree has been given first by V. I. Arnol'd in \cite{A0}. Classification of Nambu structures of smaller degree has been also considered by many authors (c.f. \cite{DN1} and references therein for the state of the art in this direction). Moreover, several global classification results have appeared in the literature, mostly for the maximal degree case (c.f. \cite{Ra}, \cite{To} and also \cite{GM}, \cite{Kis}, \cite{Mi} for relations with so called $b$-geometry, integrable systems, generalisations of KAM theorem e.t.c.). 

Despite these studies, and many others which is impossible to cite, the simplest (and undoubtedly important in terms of applications) problem of local classification of functions  with respect to a singular Nambu structure has been left untouched, at least in the authors' knowledge, whereas the non-degenerate case has been known already for almost a century, starting from the works of G. D. Birkhoff \cite{Bir} (for $r=2$), to J. Vey's isochore Morse lemma \cite{Ve} (for $r=n$), and its generalisations c.f. \cite{F0}, \cite{F}, \cite{G}. The main purpose of the paper is to fill in this gap.  

Here we will consider the classification of functions $f$ and Nambu structures $\{\cdot,\cdots,\cdot\}$ of maximal degree, in a neighbourhood of a non-singular (smooth) point of its degeneracy hypersurface $\Sigma$. Classification of functions at the singular points of the degeneracy hypersurface, as well as classification of functions with respect to Nambu structures of lower degree, are much more complicated problems and will not be considered here.

\medskip 

\noindent CONVENTION: All the objects in the paper are complex analytic (holomorphic) germs at the origin of $\mathbb{C}^{n+1}$, $n\geq 1$, unless otherwise stated. All diffeomorphisms (biholomorphisms) considered are tangent to the identity. Some of the results can be also extended to the smooth case as well up to certain modifications (as in the $C^{\infty}$- isochore Morse lemma \cite{CV} and its generalisations \cite{R}), but we will not deal with this problem here.  

\medskip

The structure of the paper is as follows: in Section 2 we present the classification of typical singularities of functions with respect to a singular Nambu structure at the smooth points of its degeneracy hypersurface. We also present some direct important corollaries: classification of generic Hamiltonian vector fields in $\frak{aff}^*(1)$ (the dual of the Lie algebra $\frak{aff}(1)$ of affine transformation of the line with its standard Poisson structure) and classification of generic 1-parameter families of Poisson structures in 3-dimensional manifolds, ``subordinated" to the given Nambu structure (to be explained briefly in the text, c.f. \cite{T} for more details). 

The proofs, given in Section 4, are direct consequences of the results presented in Section 3 which is the main part of the paper. There we present all the technical details needed for the classification of Nambu structures with respect to diffeomorphisms preserving any isolated boundary singularity $(f,\Sigma)$ (in the terminology of V. I. Arnol'd \cite{A1}), where $\Sigma$ is the degeneracy hypersurface of the Nambu structure. The main result of the section is an analog of the so-called Brieskorn-Sebastiani theorem \cite{B}, \cite{Se}, which guarantees the finiteness and freeness of the deformation module of the (dual of the) Nambu structure:
\[H''_f(\log \Sigma):=\frac{\Omega^{n+1}(\log \Sigma)}{df\wedge d\Omega^{n-1}(\log \Sigma)},\]
where we denote by $\Omega^{\bullet}(\log \Sigma)$ the complex of differential forms having logarithmic singularities along $\Sigma$.  The latter deformation module is an important object associated to any isolated boundary singularity $(f,\Sigma)$ and it can be viewed as the logarithmic analog of the well known Brieskorn module, since it provides a natural extension at the origin of the sheaves of sections of the cohomology bundle $\cup_{t\in S^*}H^n(X_t\setminus X_t\cap \Sigma;\mathbb{C})$ of the complement of the Milnor fibers $X_t\cap \Sigma$ of the restriction $f|_{\Sigma}$ of the function $f$ on $\Sigma$, inside the Milnor fibers $X_t$ of $f$. This relates the classification problem with the corresponding relative logarithmic cohomology: the functional moduli can be interpreted in terms of integrals of relative logarithmic forms  $\omega/df|_{X_t}$ along the vanishing cycles of the complement $X_t\setminus X_t\cap \Sigma$ which, as it was shown in \cite{S}, generate the homology bundle $\cup_{t\in S^*}H_n(X_t\setminus X_t\cap \Sigma;\mathbb{C})$. 

It is important to remark here that the classification problem we study is closely related, and is in fact dual, to the classification of functions and non-singular Nambu structures (non-singular $(n+1)$-mutivector fields, or dually, volume forms) on a manifold with boundary $\Sigma$ (c.f. \cite{K}). This duality is expressed in cohomological terms as a duality between the relative (co)homology $H_n(X_t,X_t\cap \Sigma;\mathbb{C})$  of a pair of Milnor fibers, and the (co)homology of the complement $H_n(X_t\setminus X_t\cap \Sigma;\mathbb{C})$,  studied extensively in \cite{S}. It extends much further to a natural duality between the corresponding Gauss-Manin connections (Picard-Lefschetz monodromies) and the eventual mixed Hodge structures in these vanishing cohomologies. In fact, taking one step further the results of the present paper (viewing for example the logarithmic Brieskorn module $H''_f(\log \Sigma)$ as a lattice sitting inside the corresponding Gauss-Manin system), it is easy to define an asymptotic mixed Hodge structure in $\cup_{t\in S^*}H^n(X_t\setminus X_t\cap \Sigma;\mathbb{C})$ as in \cite{Var}, compute the associated spectral pairs (Hodge numbers) and so on. Here we don't take this step but instead we present only those results which are intimately related with the classification problem at hand.

\section{Classification of Typical Singularities and Some Corollaries}

In classifying pairs $(\{\cdot,\cdots,\cdot\},f)$ at points $0\in \Sigma$ of the degeneracy hypersurface $\Sigma$ of the Nambu structure $\{\cdot,\cdots,\cdot\}$, one distinguishes the following (first occurring) singularity classes:

\begin{itemize}
\item[$A_0$ (non-singular case):] The hypersurface $\Sigma$ is non-singular (smooth), the function $f$ is non-singular ($df(0)\ne 0$) and it is transversal to $\Sigma$, i.e. $df\wedge d\sigma(0)\ne 0$, where $\sigma$ is an equation of $\Sigma$.
\item[$A_1$ (relative Morse case):] Both the hypersurface $\Sigma$ and the function $f$ are non-singular, but $f$ is non-transversal to $\Sigma$, $df\wedge d\sigma(0)=0$, and its restriction $f|_{\Sigma}$ on $\Sigma$ has a non-degenerate (Morse) critical point at the origin, $df|_{\Sigma}(0)=0$, $d^2f|_{\Sigma}(0)\ne 0$.
\end{itemize}

On the plane $\mathbb{C}^2$ the singularities $A_0$ and $A_1$ are the only generic singularities (all other singularities are of codimension $\geq 3$). In higher dimensions there exists one more typical singularity class (isolated) for which the hypersurface $\Sigma$ has a non-degenerate (Morse) singularity at the origin, and $f$ is non-singular and transversal to the smooth part of $\Sigma$. As it was indicated in the introduction, this case is more complicated and it will not be considered here.

As it will become apparent in the text, there is a distinctive difference in the classification problem between the 2-dimensional  case and the higher dimensional one; functional moduli appear even in the classification of non-singular pairs $A_0$. Recall (it is easy to show c.f. \cite{A0}) that a generic Poisson structure $\{\cdot,\cdot\}$ on the plane $\mathbb{C}^2$ is equivalent, in a neighbourhood of a point of its degeneracy curve $\Sigma$, to the standard Poisson structure of $\frak{aff}^*(1)$:
\begin{equation}
\label{pois2}
\{x,y\}=x.
\end{equation}

\begin{thm}
\label{thm1}
Any generic function $f$ on the Poisson plane $(\mathbb{C}^{2},\{\cdot,\cdot\})\cong \frak{aff}^*(1)$ is equivalent, by a diffeomorphism preserving (\ref{pois2}), to one of the following normal forms:
\begin{equation}
\label{nfA02}
A_0:\quad  f=\phi(y), \quad \phi(0)=0, \hspace{0.2cm} \phi'(0)=1, 
\end{equation}
\begin{equation}
\label{nfA12}
A_1: \quad f=\zeta(x+(y+\xi(x+y^2))^2), \quad \zeta(0)=0,\hspace{0.2cm} \zeta'(0)=1.
\end{equation}
The functions of one variable $\phi(t)$, (where $t=y$), and $\xi(t)$, $\zeta(\tau)$ (where $t=x+y^2$ and $\tau=x+(y+\xi(x+y^2)^2)$), are functional moduli.
\end{thm}


The functional moduli appearing in the theorem admit a geometric description in terms of composed residues of the dual Poisson structure (for the $A_0$ case), as well as its integrals along the vanishing cycles of the Milnor fiber $X_t$ of $f$, minus the two points $X_t\cap \Sigma$ (for the $A_1$ case, c.f. Section 3.2 for more details).  Another way to understand the nature of these moduli is to observe that  the classification of pairs $(\{\cdot,\cdot\},f)$ on the plane immediately implies the classification, under conjugacy, of the corresponding Hamiltonian vector fields $Z_f:=\{f,\cdot\}$. As it can be easily verified, these vector fields have non-isolated singularities along $\Sigma$ (i.e. they vanish identically along $\Sigma$), and their classification contains a-priori functional moduli (even if we forget the pair $(\{\cdot,\cdot\},f)$).

\begin{cor}
\label{hvf}
A generic Hamiltonian vector field $Z_f$ on the Poisson plane $(\mathbb{C}^2,\{\cdot,\cdot\})\cong \frak{aff}^*(1)$ can be reduced, by a diffeomorphism preserving (\ref{pois2}), to one of the following normal forms:
\[A_0: \quad Z_f=c(y)x\partial_x, \quad c(0)=1,\]
\[A_1:\quad Z_f=\frac{b(x+y^2)}{1+ya(x+y^2)}x(2y\partial_x-\partial_y), \quad  b(0)=1.\]
\end{cor} 
\begin{proof}
The case $A_0$ is immediate from normal form (\ref{nfA02}) and equation $Z_f=\{f,\cdot\}$, where we have set $c(y)=\phi'(y)$. For the case $A_1$  it is more convenient to use a different normal form than (\ref{nfA12}), which contains two functional invariants of one (the same) variable $t=x+y^2$. Indeed, starting from (\ref{nfA12}) we consider change of coordinates $y+\xi(x+y^2)\mapsto y$ which leads, after some computations (c.f. Proof of Theorem \ref{thm1} in Section 4 for more details) to the new normal form:
\[\{x,y\}=\frac{x}{1+ya(x+y^2)}, \quad f=\zeta(x+y^2),\]
with the functional invariants $a(t)$, $\zeta(t)$, $\zeta'(0)=1$.
The result then follows again by equation $Z_f=\{f,\cdot\}$ where we have set $\zeta'(x+y^2)=b(x+y^2)$. 
\end{proof}

\begin{remark}
\normalfont{It follows from this that if one is interested only in the structure of the corresponding phase portraits (orbital equivalence), the moduli appearing in the theorem can be ``killed'' by a time reparametrisation. The corresponding orbital normal forms are then the classical ones  obtained by the classification of generic pairs $(Z,\Sigma)$=(non-singular vector field, smooth curve) (c.f. \cite{Z}).}
\end{remark}

Let us consider now the higher dimensional case $n\geq 2$.  It is easy to see again (c.f. \cite{A0}) that a Nambu structure $\{\cdot,\cdots,\cdot\}$ of maximal degree can be reduced, in a neighborhood of a smooth point of its degeneracy hypersurface $\Sigma$, to the normal form:
\begin{equation}
\label{nsn}
\{x,y_1,\cdots,y_n\}=x.
\end{equation}

\begin{thm}
\label{thm2}
Any generic function $f$ in $(\mathbb{C}^{n+1},\{\cdot, \cdots, \cdot\})$, $n\geq 2$, is equivalent, by a diffeomorphism preserving (\ref{nsn}), to one of the following normal forms:
\begin{equation}
\label{nfA0n}
A_0:\quad \quad f=y_1.
\end{equation}
\begin{equation}
\label{nfA1n}
A_1: \quad f=\zeta(x+\sum_{i=1}^ny_i^2), \quad \zeta(0)=0, \hspace{0.2cm} \zeta'(0)=1.
\end{equation}
The function of one variable $\zeta(t)$ (where $t=x+\sum_{i=1}^ny_i^2$) is a functional modulus.
\end{thm}

Again here the functional modulus $\zeta(t)$ admits a geometric description in terms of integrals of the dual Nambu form along the vanishing cycles of the complement of Milnor fibers $X_t\setminus X_t\cap \Sigma$ of the boundary singularity $(f,\Sigma)$ (c.f. Section 3.1).

In contrast to the 2-dimensional case, a pair $(\{\cdot,\cdots,\cdot\},f)$ in $\mathbb{C}^{n+1}$, $n\geq 2$ does not define  anymore a Hamiltonian vector field, but instead, it defines through equation $\{f,\cdots,\cdot\}:=\{\cdot,\cdots,\cdot\}_f$, a 1-parameter family of Nambu structures of degree $n$ (one less), ``subordinated" to the given one (in the terminology of \cite{T}).  As in the 2-dimensional case, their classification is induced by the classification of the corresponding pairs $(\{\cdot,\cdots,\cdot\},f)$ and thus, Theorem \ref{thm2} above gives the normal forms of generic subordinated Nambu structures. Already for the 3-dimensional case $n=2$, the corresponding result is interesting enough:

\begin{cor}
\label{sps}
A generic Poisson structure in $\mathbb{C}^3$ subordinated to a given Nambu structure is equivalent, at a smooth point $0\in \Sigma$ of the degeneracy surface of the Nambu structure, to one of the following normal forms:
\[A_0:\quad \{x,y\}_t=0, \quad \{x,z\}_t=-x,  \quad \{y,z\}_t=0, \quad t=y.\]
\[A_1:\quad \{x,y\}_t=2xzc(t), \quad \{x,z\}_t=-2xyc(t), \quad \{y,z\}_t=xc(t), ,\quad t=x+y^2+z^2.\]
The function of one variable $c(t)$, $c(0)=1$, is a functional modulus.
\end{cor}
\begin{proof}
It follows immediately from equation $\{f,\cdot,\cdot\}=\{\cdot,\cdot\}_f$, where we have set $c(t)=\zeta'(t)$, $t=f$.
\end{proof}

\begin{remark}
\normalfont{It is well known that any Poisson structure on a manifold defines a foliation by symplectic leaves of different dimensions (c.f. \cite{W}). In both cases above the symplectic foliation consists of the points of the hypersurface $\Sigma=\{x=0\}$ where the Poisson structure identically vanishes (leaves of dimension 0), and the complement $f^{-1}(t)\setminus f^{-1}(t)\cap \Sigma$ of this hypersurface inside the fibers of the corresponding function $f$ (symplectic leaves of dimension 2).}
\end{remark}    

\section{Equivalence of Nambu Structures and the Logarithmic Brieskorn Module}

For economy in the exposition we will use throughout the paper the following:

\medskip

\noindent NOTATION: We denote by $(x,y)$ the coordinates of $\mathbb{C}^{n+1}$, where $y=(y_1,\cdots,y_n)$. We also denote by:
\begin{itemize}
\item[(-)] $\partial_y^n:=\partial_{y_1}\wedge \cdots \wedge \partial_{y_n}$ the standard multivector field in $\mathbb{C}^n(y)$.
\item[(-)] $dy^n:=dy_1\wedge \cdots \wedge dy_n$ the standard volume form in $\mathbb{C}^n(y)$.
\item[(-)] $\Omega^{\bullet}(\log \Sigma)$ the complex of logarithmic forms on $\Sigma$, i.e. those forms $a$ such that both $\sigma a$ and $\sigma da$ are holomorphic, where $\sigma$ is a local equation for $\Sigma$.  In coordinates $(x,y)$ such that $\Sigma=\{x=0\}$:
\[\Omega^{\bullet}(\log \Sigma)=\frac{dx}{x}\wedge \Omega^{\bullet-1}+\Omega^{\bullet}.\]
\item[(-)] $\Omega^{\bullet}(\Sigma)\subset \Omega^{\bullet}$ the subcomplex of holomorphic forms vanishing on the hypersurface $\Sigma$ and $\Omega^{\bullet}_{\Sigma}:=\Omega^{\bullet}/\Omega^{\bullet}(\Sigma)$ the complex of holomorphic forms on $\Sigma$. If $\Sigma=\{x=0\}$ in the coordinates above then: 
\[\Omega^{\bullet}(\Sigma)=x\Omega^{\bullet}(\log \Sigma)=dx\wedge \Omega^{\bullet-1}+x\Omega^{\bullet}.\] 
\end{itemize}

\medskip

Given a Nambu structure $\{\cdot,\cdots,\cdot\}$ of maximal degree in $\mathbb{C}^{n+1}$, or equivalently an $(n+1)$-multivector field $\Pi$, we may naturally associate to it a dual $(n+1)$-form $\omega$, and conversely, through equation $\Pi \lrcorner \omega=1$. In a fixed coordinate system $(x,y)$ such that
\[\Pi=\sigma(x,y)\partial_x\wedge \partial_y^n,\]
the dual form is then given by:
\[\omega=\frac{1}{\sigma(x,y)}dx\wedge dy^n=\sigma(x,y)(dx\wedge dy^n)^{-1}.\]
We call $\omega$ a Nambu (or Poisson for $n=1$) form. If $\Pi$ degenerates (vanishes) along the hypersurface $\Sigma=\{\sigma(x,y)=0\}$, it follows that the corresponding Nambu form will have logarithmic singularities along $\Sigma$, i.e. it will be an element of the module $\Omega^{n+1}(\log \Sigma)$. If $\Sigma$ is smooth, given say by the vanishing of the coordinate function $x$, $\Sigma=\{x=0\}$, then we can write:
\[\Pi=xg(x,y)\partial_x\wedge \partial_y^n\Longleftrightarrow \omega=\frac{1}{xg(x,y)}dx\wedge dy^n,\]
for some function $g$, $g(0)\ne 0$. We denote by $\Omega_*^{n+1}(\log \Sigma)$ the corresponding module of those logarithmic forms for which $g(0)=1$.

It follows from the above that classification of pairs $(\{\cdot,\cdots,\cdot\},f)$, or equivalently of $(\Pi,f)$, can be reduced to classification of pairs $(\omega,f)$, where $\omega$ is the dual Nambu form. To classify the latter, it turns out that it is more convenient to fix the pair $(f,\Sigma)$, where $\Sigma$ is the degeneracy hypersurface of the Nambu structure, and classify Nambu forms $\omega \in \Omega_*^{n+1}(\log \Sigma)$ by diffeomorphisms preserving the pair $(f,\Sigma)$. 

Singularities of pairs $(f,\Sigma)$ have been studied extensively by V. I. Arnol'd \cite{A1} and his collaborators (c.f. \cite{A2} for a collection of results). They are known in the literature as boundary singularities and they are defined by the conditions that both $f$ and its restriction $f|_{\Sigma}$ on $\Sigma$, have an isolated singularity at the origin (without excluding the case where $f$ might be non-singular, $df(0)\ne 0$, but non-transversal to the boundary). Below we present all the technical details needed for the classification of Nambu forms $\omega$ with respect to an isolated boundary singularity $(f,\Sigma)$. We start our exposition with the simpler (with less moduli) case $n\geq 2$ and in the next Section 3.2 we provide all the necessary modifications for the planar case $n=1$ as well.

\subsection{Case $n\geq 2$}

First we will need the following lemma, which is a logarithmic analog of the so called de Rham division lemma \cite{dR}.   

\begin{lem}
\label{logdrd}
Let $a \in \Omega^n(\log \Sigma)$ be a logarithmic form such that $df\wedge a=0$. Then there exists a logarithmic form $b\in \Omega^{n-1}(\log \Sigma)$ such that $a=df\wedge b$.
\end{lem}
\begin{proof}
Fix coordinates such that $\Sigma=\{x=0\}$, and multiply the relation $df\wedge a=0$ by $x$. We obtain $df\wedge a'=0$, where $a'=xa \in \Omega^n(\Sigma)$ is a form vanishing on $\Sigma$. Since $\Omega^n(\Sigma)\subset \Omega^n$, we know from the ordinary de Rham division lemma that there exists a form $b'\in \Omega^{n-1}$ such that $a'=df\wedge b'$. Now, since $a'|_{\Sigma}=0$ it follows that $b'|_{\Sigma}=0$ as well. Indeed, if $b'|_{\Sigma} \ne 0$, then since $df|_{\Sigma}\wedge b'|_{\Sigma}=0$, we have again by the ordinary de Rham division lemma that there exists some $\gamma_{\Sigma} \in \Omega^{n-2}_{\Sigma}$ such that $b'|_{\Sigma}=df|_{\Sigma}\wedge \gamma_{\Sigma}$. Changing now $b'$ with $b''=b'-df\wedge \gamma_{\Sigma}$ we obtain that $a'=df\wedge b'=df\wedge b''$, where now $b''\in \Omega^{n-1}(\Sigma)$, i.e. it  vanishes on $\Sigma$. Since $a'=xa$, we obtain by division with $x$ that $a=df\wedge b$, where $b=b''/x\in \Omega^{n-1}(\log \Sigma)$ is logarithmic, and this finishes the proof. 
\end{proof}

Denote now by $\mathcal{R}(f,\Sigma)$ the isotropy group of the pair $(f,\Sigma)$ (i.e. diffeomorphisms preserving the boundary singularity $(f,\Sigma)$), and by $\Theta(f,\Sigma)$ its tangent space at the identity:
\[\Theta(f,\Sigma)=\{V\in \Theta/L_V(f)=0, \hspace{0.3cm} L_V(x)\subseteq <x>\},\]
where we have chosen coordinates $(x,y)$ such that $\Sigma=\{x=0\}$. The group $\mathcal{R}(f,\Sigma)$ acts naturally on the space $\Omega^{n+1}_*(\log \Sigma)$ of Nambu forms. Denote by $\mathcal{R}(\omega)$ the orbit of $\omega$ and $T(\omega)$ its tangent space at $\omega$:
\[T(\omega)=\{L_V\omega/V\in \Theta(f,\Sigma)\}=\{L_V\omega/L_V(f)=0, \hspace{0.3cm} L_V(x)\subseteq <x>\}.\]
This submodule can be identified with the space of infinitesimally trivial deformations of $\omega$ by $\mathcal{R}(f,\Sigma)$-equivalence, and thus the quotient 
\[\mathcal{D}_{f,\Sigma}(\omega):=\frac{\Omega^{n+1}(\log \Sigma)}{T(\omega)}\]
can be identified with the module of non-trivial infinitesimal deformations.
  
\begin{lem}
\label{logdm}
The infinitesimal deformation module $\mathcal{D}_{f,\Sigma}(\omega)$ of a Nambu form $\omega \in \Omega_*^{n+1}(\log \Sigma)$ is isomorphic to the module:
\[H''_{f}(\log \Sigma):=\frac{\Omega^{n+1}(\log \Sigma)}{df\wedge d\Omega^{n-1}(\log \Sigma)}.\]
\end{lem}
\begin{proof}
It suffices to show the isomorphism:
\[T(\omega)\cong df\wedge d\Omega^{n-1}(\log \Sigma).\]
Let $V\in \Theta(f,\Sigma)$. Then 
\[0=L_V(f)\omega=df\wedge (V\lrcorner \omega),\]
and by the logarithmic de Rham division Lemma \ref{logdrd}, there exists $a\in \Omega^{n-1}(\log \Sigma)$ such that:
\[V\lrcorner \omega=df\wedge a.\]
It follows then from Cartan's formula that:
\[L_V\omega=df\wedge da.\]
Conversely, let $a\in \Omega^{n-1}(\log \Sigma)$ and consider $df\wedge a\in df\wedge \Omega^{n-1}(\log \Sigma)\subset \Omega^{n}(\log \Sigma)$. Since, as it is easy to verify, $\omega$ defines a perfect pairing between $\Theta(\Sigma)$ and $\Omega^{n}(\log \Sigma)$, there exists a vector field $V\in \Theta(\Sigma)$ such that:
\[V\lrcorner \omega=df\wedge a.\]
In fact, $V\in \Theta(f)$ as well, since multiplication by $df\wedge$ in the equation above gives:
\[df\wedge (V\lrcorner \omega)=L_V(f)\omega=0\Longrightarrow L_V(f)=0.\]
From this it follows that $L_V\omega=df\wedge da$, and the lemma is proved.   
\end{proof}

The infinitesimal deformation module $H''_f(\log \Sigma)$ introduced above is an important object associated to the pair $(f,\Sigma)$ (it is independent of $\omega$ by definition); it has a natural $\mathbb{C}\{t\}$-module structure with multiplication induced by $f=t$, and it can be viewed as a logarithmic analog of the ordinary Brieskorn module $H''_f$ associated to an isolated singularity $f$ (c.f. \cite{B}). In order to state the main theorem concerning the finiteness and freeness of this logarithmic Brieskorn module (the logarithmic analog of the so-called Brieskorn-Sebastiani theorem \cite{B},\cite{Se}), recall (c.f. \cite{A1}) that given an isolated boundary singularity $(f,\Sigma)$, the multiplicity of its critical point (the relative Milnor number) is the number 
\[\mu_{\Sigma}(f)=\mu(f)+\mu(f|_{\Sigma}),\]
i.e. the sum of Milnor numbers of the function $f$ and of its restriction $f|_{\Sigma}$ on the boundary $\Sigma$. It can be interpreted analytically as the $\mathbb{C}$-dimension of the relative Milnor module:
\[\Omega_f(\Sigma):=\frac{\Omega^{n+1}}{df\wedge \Omega^n(\Sigma)}\cong Q_{f,\Sigma}:=\frac{\mathcal{O}_{n+1}}{<x\partial_xf, \partial_{y_1}f,\cdots,\partial_{y_n}f>},\]
(where the last isomorphism is obtained by division with a volume form $dx\wedge dy^n$ in coordinates $(x,y)$ for which $\Sigma=\{x=0\}$).  Notice that division by $x$ in the above Milnor module $\Omega_f(\Sigma)$, gives also an isomorphism of $\mathbb{C}$-vector spaces:
\[\Omega_f(\Sigma)\stackrel{\cdot\frac{1}{x}}{\cong} \Omega_f(\log \Sigma):=\frac{\Omega^{n+1}(\log \Sigma)}{df\wedge \Omega^n(\log \Sigma)},\]
the term on the right being interpreted as a logarithmic Milnor module, which is again a $\mathbb{C}$-vector space of dimension $\mu_{\Sigma}(f)$.  

To interpret the relative Milnor number topologically, one may consider a standard representative $(f,\Sigma):X\rightarrow S$ of the germ $(f,\Sigma)$ (as for example in \cite{S}, see also \cite{K} for more details). Then, $\mu_{\Sigma}(f)$ is equal to the rank of the relative homology group $H_n(X_t,X_t\cap \Sigma;\mathbb{Z})$ of the pair of Milnor fibers $(X_t,X_t\cap \Sigma)$ of $(f,f|_{\Sigma})$ c.f. \cite{A1}, and by duality, it is also equal to the rank of the homology group $H_n(X_t\setminus X_t\cap \Sigma;\mathbb{Z})$ of the complement of the Milnor fiber $X_t\cap \Sigma$ of the restriction $f|_{\Sigma}$ inside the Milnor fiber $X_t$ of $f$  c.f. \cite{S}. In fact, viewing this last homology group with coefficients over the complex numbers, one easily obtains a standard Gysin-Thom short exact sequence of homological vector bundles over the punctured disc $S^*=S\setminus 0$:
\begin{equation}
\label{GT}
0\rightarrow \cup_{t\in S^*}H_{n-1}(X_t\cap \Sigma;\mathbb{C})\stackrel{\mathcal{L}}{\rightarrow} \cup_{t\in S^*}H_n(X_t\setminus X_t\cap \Sigma;\mathbb{C})\rightarrow \cup_{t\in S^*}H_{n}(X_t;\mathbb{C})\rightarrow 0,
\end{equation}
where $\mathcal{L}$ is the so-called Leray tube operator (it inflates a cycle of $X_t\cap \Sigma$ inside the complement $X_t\setminus X_t\cap \Sigma)$. It follows that the locally constant sections which generate the homology bundle $\cup_{t\in S^*}H_n(X_t\setminus X_t\cap \Sigma;\mathbb{C})$ are exactly the vanishing cycles of $X_t$ (which are the generators of $\cup_{t\in S^*}H_{n}(X_t;\mathbb{C})$) and (the image by $\mathcal{L}$ of) the vanishing cycles of the restriction $X_t\cap \Sigma$ (which are the generators of $\cup_{t\in S^*}H_{n-1}(X_t\cap \Sigma;\mathbb{C})$). Their number is then exactly equal to $\mu_{\Sigma}(f)=\mu(f)+\mu(f|_{\Sigma})$.  
\begin{thm}
\label{thm3}
The logarithmic Brieskorn module $H''_f(\log \Sigma)$ is a free $\mathbb{C}\{t\}$-module of rank $\mu_{\Sigma}(f)$:
\[H''_{f}(\log \Sigma)\cong \mathbb{C}\{t\}^{\mu_{\Sigma}(f)}.\]
Moreover, it is a natural extension at the origin of the sheaf of sections of the cohomological Milnor bundle $\cup_{t\in S^*}H^n(X_t\setminus X_t\cap \Sigma;\mathbb{C})$.
\end{thm}
\begin{proof}
Fix equation $x$ for $\Sigma$ and let $i:\Sigma \hookrightarrow \mathbb{C}^{n+1}$, $i(y)=(0,y)$ be the corresponding embedding. Consider the Poincar\'e residue short exact sequence:
\[0\rightarrow \Omega^{\bullet}\rightarrow \Omega^{\bullet}(\log \Sigma)\stackrel{R}{\rightarrow} i_*\Omega^{\bullet-1}_{\Sigma}\rightarrow 0,\]
where the residue map is defined as follows: for each $p\geq 1$, and  $\omega \in \Omega^{p}(\log \Sigma)$, let $a\in \Omega^{p-1}$ be such that $x\omega=dx\wedge a$. Then $R(\omega)=a|_{\Sigma}$.
For $p=n+1$ we have thus a short exact sequence:
\begin{equation}
\label{prses}
0\rightarrow \Omega^{n+1}\rightarrow \Omega^{n+1}(\log \Sigma)\stackrel{R}{\rightarrow} i_*\Omega^n_{\Sigma}\rightarrow 0.
\end{equation}
The morphism $R$ commutes with the differentials $d$ and, as it is easy to check, it also commutes with $df\wedge$. Moreover, it is $\mathbb{C}\{t\}$-linear (also straightforward) and thus we obtain, after passing to quotients, a short exact sequence of $\mathbb{C}\{t\}$-modules:
\begin{equation}
\label{sesbm}
0\rightarrow H''_f\rightarrow H''_f(\log \Sigma)\stackrel{R}{\rightarrow}H''_{f|_{\Sigma}}\rightarrow 0,
\end{equation} 
where the modules on the left and on the right are the ordinary Brieskorn modules of $f$ and of its restriction $f|_{\Sigma}$ respectively:
 \[H''_f:=\frac{\Omega^{n+1}}{df\wedge d\Omega^{n-1}}, \quad H''_{f|_{\Sigma}}:= \frac{\Omega^{n}_{\Sigma}}{df|_{\Sigma}\wedge d\Omega^{n-2}_{\Sigma}}.\]  
By the Brieskorn-Sebastiani theorem \cite{B}, \cite{Se} these are free $\mathbb{C}\{t\}$-modules of ranks equal to the corresponding Milnor numbers $\mu(f)$ and $\mu(f|_{\Sigma})$ respectively, and thus the module $H''_f(\log \Sigma)$ in the middle is also free of rank $\mu_{\Sigma}(f)=\mu(f)+\mu(f|_{\Sigma})$. This proves the finiteness and freeness part of the theorem. The last part is also obtained immediately as follows: sheafifying the short exact sequence (\ref{sesbm}) and taking the corresponding Gelfand-Leray residue forms $\omega/df$ (viewed as sections over $S^*$), we obtain an isomorphic short exact sequence of de Rham cohomology bundles:
\[0\rightarrow \cup_{t\in S^*}H^n_{dR}(X_t;\mathbb{C})\rightarrow \cup_{t\in S^*}H^n_{dR}(X_t\setminus X_t\cap \Sigma;\mathbb{C})\stackrel{R}{\rightarrow} \cup_{t\in S^*}H^{n-1}_{dR}(X_t\cap \Sigma;\mathbb{C})\rightarrow 0.\]
Then, since $X_t$ and $X_t\cap \Sigma$ are Stein, the de Rham integration morphism:
\[I(t):=\int_{\gamma(t)}\frac{\omega}{df},\] 
identifies the de Rham cohomology short exact sequence above, with the dual of the Thom-Gysin short exact sequence (\ref{GT}), which is exactly the short exact sequence of cohomology bundles:
 \[0\rightarrow \cup_{t\in S^*}H^n(X_t;\mathbb{C})\rightarrow \cup_{t\in S^*}H^n(X_t\setminus X_t\cap \Sigma;\mathbb{C})\rightarrow \cup_{t\in S^*}H^{n-1}(X_t\cap \Sigma;\mathbb{C})\rightarrow 0.\]
This proves the theorem.
\end{proof}

In particular, fixing a basis $\{\omega_i\}_{i=1}^{\mu_{\Sigma}(f)}$ of the logarithmic Brieskorn module $H''_f(\log \Sigma)$, the corresponding geometric sections $\{\omega_i/df\}_{i=1}^{\mu_{\Sigma}(f)}$ form a trivilisation of the cohomology bundle $\cup_{t\in S^*}H^n(X_t\setminus X_t\cap \Sigma;\mathbb{C})$. In such a fixed basis, if $\omega$ is any logarithmic form then there are uniquely defined holomorphic functions $\{c_i(t)\}_{i=1}^{\mu_{\Sigma}(f)}$ such that 
\[\omega=\sum_{i=1}^{\mu_{\Sigma}(f)}c_i(f)\omega_i \mod(df\wedge d\Omega^{n-1}(\log \Sigma)).\]
If $\gamma_j(t)\in \cup_{t\in S^*}H_n(X_t\setminus X_t\cap \Sigma;\mathbb{C})$ is an element of a basis of  locally constant sections (e.g. a vanishing cycle) then there exists a decomposition of the integral $I_j(t)=\int_{\gamma_j(t)}\omega/df$ as:
\[I_j(t)=\sum_{i=1}^{\mu_{\Sigma}(f)}c_i(t)P_{ij}(t), \quad P_{ij}(t)=\int_{\gamma_j(t)}\frac{\omega_i}{df},\]
where the matrix $P(t)=(P_{ij}(t))_{1\leq i,j\leq \mu_{\Sigma(f)}}$ is the corresponding period matrix. It follows by Cramer's rule that each function $c_i(t)$ can be expressed in terms of integrals along the vanishing cycles, as:
\[c_i(t)=\frac{\det \tilde{P}_i(t)}{\det P(t)},\]
where the matrix $\tilde{P}_i(t)$ is obtained by the period matrix $P(t)$ after replacing its $i$'th column with the vector $I(t)=(I_1(t),\cdots,I_{\mu_{\Sigma}(f)}(t))^T$.  

It is obvious from the construction above that the integrals $I(t)$, as well as the functions $c(t)=(c_1(t),\cdots,c_{n+1}(t))$, form a set of functional invariants of the form's $\omega$ $\mathcal{R}(f,\Sigma)$-equivalence class. The theorem below verifies that this set is in fact a complete set of functional invariants for the classification problem:
\begin{thm}
\label{hm}
Any two Nambu forms $\omega, \omega' \in \Omega_*^{n+1}(\log \Sigma)$ are $\mathcal{R}(f,\Sigma)$-equivalent if and only if they define the same class in the logarithmic Brieskorn module $H''_f(\log \Sigma)$. In particular, any Nambu form $\omega'$ is equivalent to its representative in the logarithmic Brieskorn module:
\[\omega=\sum_{i=1}^{\mu_{\Sigma}(f)}c_i(f)\omega_i.\]
The holomorphic functions $c_i(t)$ are functional moduli (for the fixed basis $\{\omega_i\}_{i=1}^{\mu_{\Sigma}(f)}$ of $H''_f(\log \Sigma)$). 
\end{thm}
\begin{proof}
The one direction is immediate: if the forms $\omega$ and $\omega'$ are equivalent then their Poincar\'e residues $\omega/df$, $\omega'/df$ define the same cohomology class in each fiber $H^n(X_t\setminus X_t\cap \Sigma;\mathbb{C})$ of the cohomological Milnor fibration, in a sufficiently small neighbourhood of the origin. Indeed, since the diffeomorphism $\Phi \in \mathcal{R}(f,\Sigma)$ realising the equivalence is tangent to the identity, it induces the identity on the cohomology of $X_t\setminus X_t\cap \Sigma$ with constant coefficients (because the latter is topological). It follows from this that  the diffeomorphism $\Phi$ induces also the identity in $H''_f(\log \Sigma)$. The other direction follows from an application of Moser's homotopy method: connect $\omega$, $\omega'$ by a path $\omega_s=\omega+sdf\wedge da$, $s\in [0,1]$, $a\in \Omega^{n-1}(\log \Sigma)$. Then the vector field $V_s\in \Theta(f,\Sigma)$ defined by the equation:
\[V_s\lrcorner \omega_s=a\wedge df,\]
is a solution of the homological equation:
\[L_{V_s}\omega_s=df\wedge d(-a)\]
and the time 1-map of $V_s$ is the desired diffeomorphism. 
\end{proof}

It follows from the above theorem that in order to obtain exact classification results for pairs  $(\omega,f)$ one has to be able to  find explicitly a basis of the logarithmic Brieskorn module. By Nakayama's lemma, such a basis is obtained by lifting a basis of the corresponding $\mu_{\Sigma}(f)$-dimensional $\mathbb{C}$-vector space:
\[\frac{H''_f(\log \Sigma)}{fH''_f(\log \Sigma)}\cong \frac{\Omega^{n+1}(\log \Sigma)}{df\wedge d\Omega^{n-1}(\log \Sigma)+f\Omega^{n+1}(\log \Sigma)}.\]
To obtain in turn a basis of the latter vector space can in general be a difficult problem. An exception is the case where the pair $(f,\Sigma)$ is quasihomogeneous i.e. such that both $f$ and $f|_{\Sigma}$ are quasihomogeneous with respect to the same positive rational weights $w=(w_1,\cdots,w_{n+1})\in \mathbb{Q}^{n+1}_+$. This implies in particular that the Euler vector field $E_w$ of $f$ is also tangent to $\Sigma=\{x=0\}$:
\[E_w(f)=f, \quad E_w(x)=w_1x,\] 
and thus the function $f$ belongs to its relative (resp. logarithmic) Jacobian ideal:
\begin{equation}
\label{eqqh}
f\Omega^{n+1}\subseteq df\wedge \Omega^n(\Sigma)\Longleftrightarrow f\Omega^{n+1}(\log \Sigma)\subseteq df\wedge \Omega^n(\log \Sigma).
\end{equation}

\begin{prop}
\label{p1}
Let $(f,\Sigma)$ be a quasihomogeneous boundary singularity. Then there exists an isomorphism of $\mathbb{C}$-vector spaces:
\[\frac{H''_f(\log \Sigma)}{fH''_f(\log \Sigma)}\cong \Omega_f(\log \Sigma):=\frac{\Omega^{n+1}(\log \Sigma)}{df\wedge \Omega^n(\log \Sigma)}.\]
\end{prop}
\begin{proof}
Notice first that since $df\wedge d\Omega^{n-1}(\log \Sigma)\subset df\wedge \Omega^{n}(\log \Sigma)$, there is a natural projection:
\[\frac{H''_f(\log \Sigma)}{fH''_{f}(\log \Sigma)}\stackrel{\pi}{\rightarrow} \Omega_{X_0}(\log \Sigma):=\frac{\Omega_f(\log \Sigma)}{f\Omega_f(\log \Sigma)}\cong \frac{\Omega^{n+1}(\log \Sigma)}{df\wedge \Omega^n(\log \Sigma)+f\Omega^{n+1}(\log \Sigma)},\]
where the module on the right can be interpreted as the logarithmic Tjurina module of the isolated hypersurface singularity $X_0=\{f=0\}$. Indeed, if $\Sigma=\{x=0\}$ then multiplication by $x$ induces an isomorphism:
\[\Omega_{X_0}(\log \Sigma)\stackrel{\cdot x}{\cong} \Omega_{X_0}(\Sigma):=\frac{\Omega^{n+1}}{df\wedge \Omega^n(\Sigma)+f\Omega^{n+1}}\cong \frac{\mathcal{O}_{n+1}}{<x\partial_x f,\partial_{y_1}f,\cdots, \partial_{y_n}f>+f\mathcal{O}_{n+1}},\]
where the module on the right is the module of infinitesimal deformations of the singularity $X_0$ with respect to diffeomorphisms preserving $\Sigma$ (the relative Tjurina algebra). 
It is in particular a finite dimensional $\mathbb{C}$-vector space, whose dimension we denote by $\tau_{\Sigma}(f)$ (relative Tjurina number):
\[\tau_{\Sigma}(f):=\dim_{\mathbb{C}}\Omega_{X_0}(\Sigma)=\dim_{\mathbb{C}}\Omega_{X_0}(\log \Sigma).\] 
Thus we have obtained a short exact sequence of finite dimensional $\mathbb{C}$-vector spaces:
\begin{equation}
\label{qhses}
0\rightarrow \frac{df\wedge \Omega^n(\log \Sigma)}{df\wedge d\Omega^{n-1}(\log \Sigma)+f\Omega^{n+1}(\log \Sigma)}\rightarrow \frac{H''_f(\log \Sigma)}{fH''_f(\log \Sigma)}\stackrel{\pi}{\rightarrow}\Omega_{X_0}(\log \Sigma)\rightarrow 0,
\end{equation}
where the term on the left is of dimension $q_{\Sigma}(f):=\mu_{\Sigma}(f)-\tau_{\Sigma}(f)$. The proof now is concluded by the fact that for a quasihomogeneous boundary singularity $(f,\Sigma)$,  the function $f$ belongs to its relative Jacobian ideal (by equations (\ref{eqqh}) above), and in particular $q_{\Sigma}(f)=0\Leftrightarrow \mu_{\Sigma}(f)=\tau_{\Sigma}(f)$. 
\end{proof}  

\begin{remark}
\label{pl}
\normalfont{It is easy to show in fact (using a relative version of the Poincar\'e-Dulac theorem with respect to $\Sigma=\{x=0\}$) that the converse statement is also true, i.e. the following relative version of the well known K. Saito's theorem \cite{Sa} holds: a boundary singularity $(f,\Sigma)$ is quasihomogeneous if and only if $q_{\Sigma}(f)=0$.  This is equivalent in turn to the acyclicity of the complex $\widetilde{\Omega}_{X_0}^{\bullet}(\log \Sigma)$ which is the torsion-free part of the complex $\Omega^{\bullet}_{X_0}(\log \Sigma):=\Omega^{\bullet}_f(\log \Sigma)/f\Omega^{\bullet}_f(\log \Sigma)$ of differential forms on $X_0$ with logarithmic singularities along $\Sigma$ (notice that the last term of the latter complex is nothing but the logarithmic Tjurina module introduced above). Indeed, if $X_0$ is reduced, and if we denote by $X_0^*=X_0\setminus \{0\}$ its smooth part, then multiplication by $df\wedge$ identifies the last cohomology $H^n(\widetilde{\Omega}^{\bullet}_{X_0}(\log \Sigma))$ with the left term of the short exact sequence (\ref{qhses}) above:
\[df\wedge H^n(\widetilde{\Omega}^{\bullet}_{X_0}(\log \Sigma)):=df\wedge (\frac{\Omega^n(\log \Sigma)}{d\Omega^{n-1}(\log \Sigma)+\Omega^{n}(X_0^*,\log \Sigma)})\cong \]
\[\cong \frac{df\wedge \Omega^{n}(\log \Sigma)}{df\wedge d\Omega^{n-1}(\log \Sigma)+f\Omega^{n+1}(\log \Sigma)},\]
the last isomorphism been obtained by the obvious Poincar\'e residue short exact sequence of subcomplexes:
\begin{equation}
\label{prs}
0\rightarrow \Omega^{\bullet}(X_0^*)\rightarrow \Omega^{\bullet}(X_0^*,\log \Sigma)\stackrel{R}{\rightarrow}i_*\Omega^{\bullet-1}_{\Sigma}(X_0^*)\rightarrow 0.
\end{equation}
In fact, more is true: taking the quotient of this short exact sequence with the ordinary Poincar\'e residue short exact sequence, we obtain a short exact sequence of forms on $X_0$:
\begin{equation}
\label{sestf}
0\rightarrow \widetilde{\Omega}^{\bullet}_{X_0}\rightarrow \widetilde{\Omega}^{\bullet}_{X_0}(\log \Sigma)\stackrel{R}{\rightarrow}i_*\widetilde{\Omega}^{\bullet-1}_{X_0\cap \Sigma}\rightarrow 0,
\end{equation}
which induces in turn a long exact sequence in cohomology, whose only non-zero terms are (according to G. M Greuel for example, c.f. \cite{Gr}):
\begin{equation}
\label{sestfcoh}
0\rightarrow H^{n}(\widetilde{\Omega}^{\bullet}_{X_0}) \rightarrow H^{n}(\widetilde{\Omega}^{\bullet}_{X_0}(\log  \Sigma))\rightarrow H^{n-1}(\widetilde{\Omega}^{\bullet}_{X_0\cap \Sigma})\rightarrow 0.
\end{equation}
Again by Greuel \cite{Gr}, the terms on the right and on the left are of $\mathbb{C}$-dimensions $q(f|_{\Sigma})=\mu(f|_{\Sigma})-\tau(f|_{\Sigma}))$ and $q(f)=\mu(f)-\tau(f)$ respectively, and thus we obtain:
\[\dim_{\mathbb{C}}H^{n}(\widetilde{\Omega}^{\bullet}_{X_0}(\log  \Sigma))=q_{\Sigma}(f)=q(f)+q(f|_{\Sigma}).\]
This implies that indeed a boundary singularity $(f,\Sigma)$ is quasihomogeneous if and only if 
\[q_{\Sigma}(f)=0\Longleftrightarrow q(f)=q(f|_{\Sigma})=0.\] }
\end{remark}

Coming back to the problem of obtaining a basis for the logarithmic Brieskorn module, we conclude that for quasihomogeneous boundary singularities $(f,\Sigma)$, such a basis can be constructed as follows: take a basis $\{e_i(x,y)\}_{i=1}^{\mu_{\Sigma}(f)}$ of the local  algebra $Q_{f,\Sigma}=\mathcal{O}_{n+1}/<x\partial_xf,\partial_{y_i}(f)>$ and multiply by the standard Nambu form $x^{-1}dx\wedge dy^n$ to obtain a basis $\{\omega_i=e_i(x,y)x^{-1}dx\wedge dy^n\}_{i=1}^{\mu_{\Sigma}(f)}$ of the module $\Omega^{n+1}_f(\log \Sigma)$. Then by the isomorphism $\pi$ of Proposition \ref{pl} this defines also a basis of $H''_{f}(\log \Sigma)/fH''_f{\log \Sigma}$, which lifts, by Nakayama's lemma, to a $\mathbb{C}\{f\}$-basis of $H''_{f}(\log \Sigma)$.

\begin{example}
\label{sn}
\normalfont{All simple boundary singularities in Arnol'd's list (c.f. \cite{A2}) are quasihomogeneous and thus, by Theorem \ref{hm}, any Nambu form is equivalent to the normal form:
\[\omega=\frac{\sum_{i=1}^{\mu_{\Sigma}(f)}c_i(f)e_i(x,y)}{x}dx\wedge dy^n,\]
where the $e_i$'s form a basis of the local algebra $Q_{f,\Sigma}$. Below we give as an example the normal forms for the $A_k$, $B_k$, $C_k$ series and $F_4$ (the $D_k$ and $E_{6,7,8}$ singularities can be computed in the same way).
\[A_k: \quad \omega=\frac{\sum_{i=0}^{k-1}c_i(f)y^i}{x}dx\wedge dy^n,\quad f=x+y_1^{k+1}+Q,\quad k\geq 1\]
\[B_k: \quad \omega=\frac{\sum_{i=0}^{k-1}c_i(f)x^i}{x}dx\wedge dy^n,\quad f=x^k+y_1^2+Q,\quad k\geq 2\]
\[C_k: \quad \omega=\frac{\sum_{i=0}^{k-1}c_i(f)y^i}{x}dx\wedge dy^n,\quad f=xy_1+y_1^k+Q,\quad k\geq 2\]
\[F_4: \quad \omega=\frac{c_0(f)+c_1(f)x+c_2(f)y+c_3(f)xy}{x}dx\wedge dy^n,\quad f=x^2+y_1^3+Q,\]
where $Q=\sum_{i=2}^{n+1}y_i^2$. All the functions $c_i(t)$ appearing above are functional moduli, $c_0(0)=1$.}
\end{example}


\subsection{Case $n=1$}

The results of the previous section cannot be transferred to the $2$-dimensional case without modifications, the main reason being that by definition $\Omega^0(\log \Sigma):=\Omega^0$. Despite this fact, the corresponding finiteness and freeness theorem for the logarithmic Brieskorn module:
\[H''_f(\log \Sigma):=\frac{\Omega^2(\log \Sigma)}{df\wedge d\Omega^0},\] 
still holds true. Before we prove this, let us see how it can be identified with the infinitesimal deformation module $\mathcal{D}_{f,\Sigma}(\omega)$ of a Poisson form $\omega \in \Omega^2_*(\log \Sigma)$. For this we will need first the following:

\begin{lem}
\label{theta}
Any vector field $V\in \Theta(f,\Sigma)$ on the plane $\mathbb{C}^2$ vanishes identically along the curve $\Sigma$. 
\end{lem} 
\begin{proof}
Fix coordinates $(x,y)$ such that $\Sigma=\{x=0\}$. Since $V\in \Theta(f)$, it follows that $V$ is a multiple of the Hamiltonian vector field of $f$, i.e. there exists a function $v$ such that:
\[V=v(x,y)(\partial_yf\partial_x-\partial_xf\partial_y).\]
Since $V\in \Theta(\Sigma)$ it follows that $V(x)=v(x,y)\partial_yf\subset <x>$, and since $\partial_yf|_{x=0}\not\equiv 0$ ($f$ has at most isolated singularity on $\Sigma$), we obtain that $v(0,y)\equiv 0$, which proves the claim.
\end{proof}

Using this lemma we can now prove the $2$-dimensional analog of Lemma \ref{logdm}. 

\begin{lem}
\label{logdm2}
The infinitesimal deformation module $\mathcal{D}_{f,\Sigma}(\omega)$ of a Poisson form $\omega \in \Omega_*^{2}(\log \Sigma)$ with respect to diffeomorphisms preserving the boundary singularity $(f,\Sigma)$ is isomorphic to the logarithmic Brieskorn module:
\[H''_{f}(\log \Sigma):=\frac{\Omega^2(\log \Sigma)}{df\wedge d\Omega^0}.\]
\end{lem} 
\begin{proof}
Again, fix coordinates $(x,y)$ such that $\Sigma=\{x=0\}$. It suffices to show that the tangent space $T(\omega)=\{L_V\omega/V\in \Theta(f,\Sigma)\}$ can be identified with the submodule $df\wedge d\Omega^0\subset \Omega^2(\log \Sigma)$.
Let $V\in \Theta(f,\Sigma)$. Then 
\[0=L_V(f)\omega=df\wedge (V\lrcorner \omega).\]
Since $V$ vanishes on $\Sigma$, the 1-form $V\lrcorner \omega$ is in fact holomorphic, and by the ordinary de Rham division lemma, there exists a function $h\in \Omega^0$ such that:
\[V\lrcorner \omega=hdf.\]
It follows then from Cartan's formula that:
\[L_V\omega=df\wedge d(-h).\]
Conversely, let $h\in \Omega^0$ and consider the holomorphic 1-form $hdf\in \Omega^{1}$. Since $\omega$ defines a perfect pairing between $\Theta^T(\Sigma)$ and $\Omega^{1}$ (where we denote $\Theta^T(\Sigma)$ the submodule of vector fields vanishing on $\Sigma$), there exists a vector field $V\in \Theta^T(\Sigma)$ such that:
\[V\lrcorner \omega=hdf.\]
In fact, $V\in \Theta(f)$ as well, since multiplication by $df\wedge$ in the equation above gives:
\[df\wedge (V\lrcorner \omega)=L_V(f)\omega=0\Longrightarrow L_V(f)=0.\]
From this it follows that $L_V\omega=df\wedge d(-h)$, and the lemma is proved.   
\end{proof}

We can now state the main theorem which is the 2-dimensional analog of the logarithmic Brieskorn-Sebastiani Theorem \ref{thm3}.

\begin{thm}
\label{ff2}
The logarithmic Brieskorn module $H''_f(\log \Sigma)$ is a free $\mathbb{C}\{t\}$-module of rank $\mu_{\Sigma}(f)+1=\mu(f)+(\mu(f|_{\Sigma})+1)$:
\[H''_f(\log \Sigma)\cong \mathbb{C}\{t\}^{\mu_{\Sigma}(f)+1}.\]
Moreover, it is a natural extension at the origin of the cohomological Milnor bundle $\cup_{t\in S*} H^1(X_t\setminus X_t\cap \Sigma;\mathbb{C})$ (over a sufficiently small punctured neighborhood of the origin $0\in \mathbb{C}$).
\end{thm}
\begin{proof}
Consider the direct image through $f$ of the Poincar\'e residue short exact sequence 
\begin{equation}
\label{sesbm2}
0\rightarrow \frac{\Omega^2}{df\wedge d\Omega^0}\rightarrow \frac{\Omega^2(\log \Sigma)}{df\wedge d\Omega^0}\stackrel{R}{\rightarrow} i_*\Omega^1_{\Sigma}\rightarrow 0.
\end{equation} 
By the Weierstrass preparation theorem the module $\Omega^1_{\Sigma}$ becomes, through the mapping $f\circ i:=f|_{\Sigma}:\Sigma \rightarrow \mathbb{C}$, a free $\mathbb{C}\{t\}$-module of rank equal to the degree of $f|_{\Sigma}=t$:
\[H''_{f|_{\Sigma}}:=(f|_{\Sigma})_*\Omega^1_{\Sigma}\cong f_*i_*\Omega^1_{\Sigma}\cong \mathbb{C}\{t\}^{\text{deg}(f|_{\Sigma})}=\mathbb{C}\{t\}^{\mu(f|_{\Sigma})+1}.\]
Thus, we obtain again a short exact sequence of $\mathbb{C}\{t\}$-modules:
\[0\rightarrow H''_f\rightarrow H''_f(\log \Sigma)\stackrel{R}{\rightarrow}H''_{f|_{\Sigma}}\rightarrow 0.\]
It follows from this that the module in the middle is indeed a free $\mathbb{C}\{t\}$-module of rank $\mu_{\Sigma}(f)+1=\mu(f)+(\mu(f|_{\Sigma})+1)$. The rest of the proof is exactly the same as in the higher dimensional case with no modifications (notice that the cohomology $H^1(X_t\setminus X_t\cap \Sigma;\mathbb{C})$ is indeed a vector space of dimension $\mu_{\Sigma}(f)+1$, due to the short exact sequence:
\[0\rightarrow H^1(X_t;\mathbb{C})\rightarrow H^1(X_t\setminus X_t\cap \Sigma;\mathbb{C})\rightarrow H^0(X_t\cap \Sigma;\mathbb{C})\rightarrow 0,\]
where the term on the left is of dimension $\mu(f)$ and the term on the right of dimension $\mu(f|_{\Sigma})+1$.)
\end{proof}

\begin{thm}
\label{hm2}
Any two Poisson forms $\omega, \omega' \in \Omega_*^2(\log \Sigma)$ are equivalent if and only if they define the same class in the logarithmic Brieskorn module $H''_f(\log \Sigma)$. In particular, any Poisson form $\omega'$ is equivalent to its representative in the logarithmic Brieskorn module:
\[\omega=\sum_{i=1}^{\mu_{\Sigma}(f)+1}c_i(f)\omega_i.\]
The functions $c_i(t)$ are functional invariants (for the fixed basis $\{\omega_i\}_{i=1}^{\mu_{\Sigma}(f)+1}$).
\end{thm}
\begin{proof}
The proof is exactly the same as in Theorem \ref{hm} with no modification.
\end{proof}

Again, in order to obtain exact classification results for pairs $(\omega,f)$ on the plane we need to find explicitely a basis of the logarithmic Brieskorn module $H''_f(\log \Sigma)$, and in particular of the $\mathbb{C}$-vector space:
\[\frac{H''_f(\log \Sigma)}{fH''_f(\log \Sigma)}\cong \frac{\Omega^2(\log \Sigma)}{df\wedge d\Omega^0+f\Omega^2(\log \Sigma)}.\]
For quasihomogeneous boundary singularities $(f,\Sigma)$ this can be done in analogy with the higher dimensional case. Notice though that on the plane, quasihomogeneity of a boundary singularity is equivalent to the quasihomogeneity of the function $f$ only (i.e. forgetting $\Sigma$), due to the fact that the restriction $f|_{\Sigma}$ is always quasihomogeneous (as a function of 1-variable). In particular, as it is easy to verify, the following formula holds true:
\[q_{\Sigma}(f):=\mu_{\Sigma}(f)-\tau_{\Sigma}(f)=\mu(f)-\tau(f)=:q(f).\]
With this in mind we can now state the following analog of Proposition \ref{pl}:

\begin{prop}
\label{pl2}
Let $(f,\Sigma)$ be a quasihomogeneous singularity in $\mathbb{C}^2$. Then there is an isomorphism of  $\mathbb{C}$-vector spaces:
\[\frac{H''_f(\log \Sigma)}{fH''_f(\log \Sigma)}\cong \Omega_f(\log \Sigma)\oplus \mathbb{C}.\]
\end{prop}
\begin{proof}
Again by the fact that $df\wedge d\Omega^0\subset df\wedge \Omega^1(\log \Sigma)$ we obtain a short exact sequence of $\mathbb{C}$-vector spaces:
\[0\rightarrow H^1(\widetilde{\Omega}^{\bullet}_{X_0}(\log \Sigma))\stackrel{df\wedge}{\rightarrow} \frac{H''_f(\log \Sigma)}{fH''_f(\log \Sigma)}\rightarrow \Omega_{X_0}(\log \Sigma)\rightarrow 0,\]
where the term on the right is the logarithmic Tjurina module:
\[\Omega_{X_0}(\log \Sigma):=\frac{\Omega^2(\log \Sigma)}{df\wedge \Omega^1(\log \Sigma)+f\Omega^2(\log \Sigma)},\]
of dimension $\tau_{\Sigma}(f)$, and the term on the left is the first cohomology of the torsion-free complex $\widetilde{\Omega}^{\bullet}_{X_0}(\log \Sigma)$:
\[H^1(\widetilde{\Omega}^{\bullet}_{X_0}(\log \Sigma))=\frac{\Omega^1(\log \Sigma)}{d\Omega^0+\Omega^{1}(X_0^*,\log \Sigma)}\stackrel{df\wedge}{\cong}\frac{df\wedge \Omega^1(\log \Sigma)}{df\wedge d\Omega^0+f\Omega^2(\log \Sigma)},\]
with the last isomorphism been obtained again by the Poincar\'e residue short exact sequence (\ref{prs}) for $n=1$. Thus, it's $\mathbb{C}$-dimension is: 
\[\dim_{\mathbb{C}}H^1(\widetilde{\Omega}^{\bullet}_{X_0}(\log \Sigma))=q_{\Sigma}(f)+1=(\mu_{\Sigma}(f)+1)-\tau_{\Sigma}(f)=\mu(f)-\tau(f)+1.\]
The proof follows by the fact that for quasihomogeneous $f$, $q(f)=\mu(f)-\tau(f)=0$.
\end{proof}

\begin{remark}
\normalfont{As in the higher dimensional case we can compute the cohomology of the complex $\widetilde{\Omega}^{\bullet}_{X_0}(\log \Sigma)$ by taking the long exact sequence in cohomology (\ref{sestfcoh}), which for the planar case reads as:
\[0\rightarrow H^1(\widetilde{\Omega}^{\bullet}_{X_0})\rightarrow H^1(\widetilde{\Omega}^{\bullet}_{X_0}(\log \Sigma))\rightarrow H^0(\widetilde{\Omega}^{\bullet}_{X_0\cap \Sigma})\rightarrow 0.\]
The term on the left is of dimension $q(f)=\mu(f)-\tau(f)$ (again by Greuel \cite{Gr}) and the term on the right is 1-dimensional (obvious).  We conclude that indeed:
\[\dim_{\mathbb{C}}H^1(\widetilde{\Omega}^{\bullet}_{X_0}(\log \Sigma))=q_{\Sigma}(f)+1=q(f)+1.\]}
\end{remark}

It follows from the above that in order to obtain a basis for the logarithmic Brieskorn module $H''_{f}(\log \Sigma)$ for quasihomogeneous $(f,\Sigma)$, one has to choose a basis $\{e_i(x,y)\}_{i=1}^{\mu_{\Sigma}(f)}$ of the local algebra $Q_{f,\Sigma}=\mathcal{O}/<x\partial_xf,\partial_yf>$, as well as an element $e(x,y)$ which belongs to the ideal $<x\partial_xf,\partial_yf>$, but does not belong to the ideal $<f>$. Then the element $e(x,y)x^{-1}dx\wedge dy$ is a generator of $df\wedge H^1(\widetilde{\Omega}^{\bullet}_{X_0}(\log \Sigma))$, and along with the elements $\{\omega_i=e_i(x,y)x^{-1}dx\wedge dy\}_{i=1}^{\mu_{\Sigma}(f)}$, spans the vector space $H''_f(\log \Sigma)/fH''_f(\log \Sigma)$. By Nakayama's lemma, these lift to the desired $\mathbb{C}\{f\}$-basis of $H''_f(\log \Sigma)$.
  
\begin{example}
\label{s2}
\normalfont{The corresponding normal forms of Poisson structures for the non-singular $A_0$ and the simple boundary singularities $A_k$, $B_k$, $C_k$, $F_4$ are the following:
\[A_0:\quad \omega=\frac{\psi(f)}{x}dx\wedge dy, \quad f=y\]
\[A_k: \quad \omega=\frac{\sum_{i=0}^{k-1}c_i(f)y^i+\psi(f)y^k}{x}dx\wedge dy,\quad f=x+y^{k+1},\quad k\geq 1\]
\[B_k: \quad \omega=\frac{\sum_{i=0}^{k-1}c_i(f)x^i+\psi(f)y}{x}dx\wedge dy,\quad f=x^k+y^2,\quad k\geq 2\]
\[C_k: \quad \omega=\frac{\sum_{i=0}^{k-1}c_i(f)y^i+\psi(f)x}{x}dx\wedge dy,\quad f=xy+y^k,\quad k\geq 2\]
\[F_4: \quad \omega=\frac{c_0(f)+c_1(f)x+c_2(f)y+c_3(f)xy+\psi(f)y^2}{x}dx\wedge dy,\quad f=x^2+y^3.\]
All the functions $c_i(t)$, $\psi(t)$ appearing above are functional moduli, $c_0(0)=1$, and for the $A_0$ case $\psi(0)=1$ as well. For the $A_0$ case the function $\psi(f)$ admits a simple geometric meaning; it is the composed residue of $\omega$:
\[\psi(f)=R(\frac{\omega}{df})=\frac{R(\omega)}{df}.\]}
\end{example}

\section{Proofs of Theorems \ref{thm1}, \ref{thm2}}

\subsection{Proof of Theorem \ref{thm1}}

\noindent Case $A_0$: Consider normal form $A_0$ of Example \ref{s2} and let $\phi(t)$ be a function, $\phi(0)=0$, $\phi'(0)=1$, such that $\psi(t)=\phi'(t)$. Then, the change of coordinates $y\mapsto \phi(y)$ brings the pair $(\omega,f)$ to the desired normal form:
\[\omega=\frac{1}{x}dx\wedge dy, \quad f=\phi(y).\]  
\qed

\medskip

\noindent Case $A_1$: Start again with normal form $A_1$ of Example \ref{s2} and consider a diffeomorphism of the form:
\[(x,y)\stackrel{\Psi}{\longmapsto} (xv(f)),yv^{1/2}(f)),\]
for some function $v$, $v(0)=1$. The diffeomorphism $\Psi$ preserves $\Sigma=\{x=0\}$, sends $f$ to the new function $\Psi^*f=fv(f)=\zeta(f)$, $\zeta(0)=0$, $\zeta'(0)=1$, and sends the Poisson form  $\omega$ to the new form:
\[\Psi^*\omega=\frac{yv^{1/2}(f)\psi(f)+c_0(f)}{x}(2f(v^{1/2}(f))'+v^{1/2}(f))dx\wedge dy.\]
Choose now $v(t)$ as the solution of the initial value problem:
\[2t(v^{1/2}(t))'+v^{1/2}(t)=c_0^{-1}(t), \quad v^{1/2}(0)=1.\]
Setting $a(t)=v^{1/2}(t)\psi(t)c_0^{-1}(t)$ we obtain:
\[\omega=\frac{1+ya(x+y^2)}{x}dx\wedge dy, \quad f=\zeta(x+y^2),\]
which can be written as:
\[\omega=\frac{1}{x}dx\wedge d(y+\xi(x+y^2)), \quad f=\zeta(x+y^2),\]
for some function $\xi(t)$ such that $\xi'(t)=a(t)/2$. The final change of coordinates $y\mapsto y+\xi(x+y^2)$ brings the pair to the desired normal form:
\[\omega=\frac{1}{x}dx\wedge dy, \quad f=\zeta(x+(y+\xi(x+y^2)^2))\]
and the theorem is proved. 
\qed

\subsection{Proof of Theorem \ref{thm2}}

\noindent Case $A_0$: It follows immediately by Theorem \ref{hm} and the fact that $H''_f(\log \Sigma)=\{0\}$. 
\qed

\medskip

\noindent Case $A_1$: The proof is the same (easier) as for the $A_1$ case in dimension 2 presented above. Briefly, consider again normal form $A_1$ of Example \ref{sn}, and consider a diffeomorphism of the form:
\[(x,y_1,\cdots,y_n)\stackrel{\Psi}{\longmapsto} (xv(f),y_1v^{1/2}(f),\cdots,y_nv^{1/2}(f)),\]
for some function $v$ with $v(0)=1$. The diffeomorphism $\Psi$ preserves $\Sigma=\{x=0\}$, sends $f$ to some new function $\Psi^*f=fv(f)=\zeta(f)$, where $\zeta(0)=0$, $\zeta'(0)=1$, and sends $\omega$ to 
\[\Psi^*\omega=\frac{c_0(f)}{x}(\frac{2}{n}t(v^{n/2}(t))'+v^{n/2}(t))dx\wedge dy^n.\] 
Choosing now $v$ as the solution of the initial value problem:
\[\frac{2}{n}t(v^{n/2}(t))'+v^{n/2}(t)=c_0^{-1}(t), \quad v^{n/2}(0)=1,\]
we obtain the desired normal form:
\[\omega=\frac{1}{x}dx\wedge dy^n, \quad f=\zeta(x+\sum_{i=1}^ny_i^2),\]
which finishes the proof of the theorem.
\qed




\section*{Acknowledgements}

This research has been supported by the S\~ao Paulo Research Foundation (FAPESP), grant No: 2017/23555-9


\begin{thebibliography}{55}

\bibitem{A0} Arnol'd V. I., \textit{Poisson Structures on the Plane and other Powers of Volume Forms}, J. Math. Sci. 47:2509, (1989), Translated from Trudy Seminara imeni I. G. Petrovskogo, 12, (1987), 37-46. 
\bibitem{A1} Arnol'd V. I., \textit{Critical Points of Functions on a Manifold with Boundary, The simple Lie Grous $B_k$, $C_k$ and $F_4$ and Singularities of Evolutes}, Russian Mathematical Surveys, 33:5, (1978), 99-116
\bibitem{A2} Arnol'd V. I., V. V. Goryunov, O. V. Lyasko, V. A. Vasil'ev, \textit{Singularity Theory I \& II}, Encyclopaedia of Mathematical Sciences, Dynamical Systems VI \& VIII, Springer-Verlag, 6 \& 39, (1993)
\bibitem{Bir} Birkhoff G. D., \textit{Dynamical Systems}, Col. Publ., A. M. S., (1927)   
\bibitem{B} Brieskorn E., \textit{Die Monodromie der Isolierten Singularit\"aten von Hyperfl\"achen}, Manuscripta Math., 2, (1970),  103-161
\bibitem{CV} Colin de Verdi\`ere Y., Vey J., \textit{Le Lemme de Morse Isochore}, Topology, 18, (1979), 283-293
\bibitem{dR} De Rham G., \textit{Sur la Division de Formes et de Courents par une Forme Lin\'eaire}, Comment. Math. Helv., 28, (1954), 346-352 
\bibitem{DN1} Dufour J. P., Zung N. T., \textit{Poisson Structures and their Normal Forms}, Birkh\"auser, Progress in 
Mathematics, 242, (2005)
\bibitem{F0} Fran\c{c}oise J. P., \textit{Mod\`ele Local Simultan\'e d'une Fonction et une Forme de Volume}, Ast\'erisque S. M. F., 59-60, (1978), 119-130 
\bibitem{F} Fran\c{c}oise J. P., \textit{Relative Cohomology and Volume Forms}, Singularities, Banach Center Publications, 20, (1988), 207-222
\bibitem{G} Garay M. D., \textit{An Isochore Versal Deformation Theorem}, Topology, 43 (2004) 1081-1088
\bibitem{Gr} G. M. Greuel, \textit{Der Gauss-Manin-Zusammenhang Isolierter Singularit\"aten von Vollst\"andigen Durchschnitten}, Math. Ann., (1975), 235-266 
\bibitem{GM} Guillemin V., Miranda E., Pires A. R., \textit{Symplectic and Poisson Geometry on $b$-Manifolds},  Adv.  Math., 264, (2014), 864-896
\bibitem{Kis} Kiesenhofer A., Miranda E., Scott G., \textit{Action-Angle Variables and a KAM Theorem for $b$-Poisson Manifolds}, J. de Math. Pures et Appl., 105, (2016), 66-85 
\bibitem{K} Kourliouros K., \textit{Gauss-Manin Connections for Boundary Singularities and Isochore Deformations}, Dem. Mat, 48, 2, (2015), 250-288
\bibitem{Mi} Miranda E., Planas A., \textit{Classification of $b^m$-Nambu Structures of Top-Degree}, Comptes Rendus Mathematique, 356, 1, (2018), 92-96
\bibitem{N} Nambu Y., \textit{Generalised Hamiltonian dynamics}, Phys. Rev. D. 7, (1973), 2405–2412
\bibitem{Ra} Radko O., \textit{A Classification of Topologically Stable Poisson Structures on a Compact Oriented Surface}, J. Symplectic Geom., 1, (2002), 523-542
\bibitem{R} Roche C. A,. \textit{Real Relative Cohomology of Finite Codimension Germs}, Bull. A. M. S., 7, 3, (1982), 596-598
\bibitem{Sa} K. Saito, \textit{Quasihomogene Isolierte Singularit\"aten von Hyperfl\"achen}, Invent. Math. 14, (1971), 123-142
\bibitem{Se} Sebastiani M., \textit{Preuve d'une Conjecture de Brieskorn}, Manuscripta Math., 2 (1970), 301-308
\bibitem{S} Szpirglas A., \textit{Singularit\'es de Bord: Dualit\'e, Formules de Picard Lefschetz Relatives et Diagrammes de Dynkin}, Bull. Soc. Math. Fr., 118, 4,  (1990), 451-486
\bibitem{T} Takhtajan L., \textit{On Foundation of the Generalised Nambu Mechanics}, Comm. Math. Phys., 160, (1994), 295–315
\bibitem{To} Torres D. M., \textit{Global Classification of Generic Multi-Vector Fields of Top-Degree}, Journal of the LMS, 69, 3, (2004), 751-766
\bibitem{Var} Varchenko A., \textit{Asymptotic Hodge Structure in the Vanishing Cohomology}, Mathematics of the USSR-Izvestiya, 18(3):469, (1982) 
\bibitem{Ve} Vey J., \textit{Sur le Lemme de Morse}, Inventiones math. 40,  (1977), 1-9
\bibitem{W} Weinstein A. \textit{The Local Structure of Poisson Manifolds}, J. Diff. Geom., 18, (1983), 523-557
\bibitem{Z} Zhitomirskii M., \textit{Differential Forms and Vector Fields with a Manifold of Singular Points}, Matem. Contem. 5, (1993), 205-216
\end{thebibliography}
\end{document}